\documentclass[11pt,noinfoline]{article}   
\RequirePackage{times}
\RequirePackage[OT1]{fontenc}
\RequirePackage[aap,amsthm,amsmath]{imsart}
\RequirePackage{amssymb,amscd,mathrsfs,eucal}

\startlocaldefs
\def\journal@id{~}
\def\journal@name{~}
\def\journal@url{~}

\newtheorem{thm}{Theorem}[section] 
\newtheorem{cor}[thm]{Corollary}  
\newtheorem{lem}[thm]{Lemma}
\newtheorem{prop}[thm]{Proposition}
\newtheorem{conj}[thm]{Conjecture}
\theoremstyle{definition}
\newtheorem{defn}[thm]{Definition}
\theoremstyle{remark}
\newtheorem{rem}[thm]{Remark}

\numberwithin{equation}{section}
\endlocaldefs

\begin{document}

\begin{frontmatter}

\title{On the exchange of intersection and
supremum of $\sigma$-fields in filtering theory\protect\thanksref{T1}}
\runtitle{Exchange of intersection and supremum}
\thankstext{T1}{This work was partially supported by NSF grant 
DMS-1005575.}

\begin{aug}
\author{\fnms{Ramon} \snm{van Handel}\ead[label=e2]{rvan@princeton.edu}}
\runauthor{Ramon van Handel}
\affiliation{Princeton University}
\address{Sherrerd Hall 227,\\
Princeton University,\\
Princeton, NJ 08544, USA. \\ \printead{e2}}
\end{aug}

\begin{abstract}
We construct a stationary Markov process with trivial tail $\sigma$-field
and a nondegenerate observation process such that the corresponding
nonlinear filtering process is not uniquely ergodic.  This settles in the
negative a conjecture of the author in the ergodic theory of nonlinear
filters arising from an erroneous proof in the classic paper of H.\ 
Kunita
(1971), wherein an exchange of intersection and supremum of
$\sigma$-fields is taken for granted.
\end{abstract}

\begin{keyword}[class=AMS]
\kwd{37A25}                     
\kwd{37A50}                     
\kwd{60F20}                     
\kwd{60G35}                     
\kwd{60J05}                     
\kwd{60K37}                     
\end{keyword}

\begin{keyword}
\kwd{nonlinear filtering}
\kwd{unique ergodicity}
\kwd{tail $\sigma$-field}
\kwd{exchange of intersection and supremum}
\kwd{random walk in random scenery}
\end{keyword}

\end{frontmatter}

\section{Introduction and main result}
\label{sec:intro}

Let $E$ and $F$ be Polish spaces, and consider an $E\times F$-valued
stochastic process $(X_k,Y_k)_{k\in\mathbb{Z}}$ with the following 
properties:
\begin{enumerate}
\item $(X_k,Y_k)_{k\in\mathbb{Z}}$ is a stationary Markov process.
\item There exist transition kernels $P$ from $E$ to $E$ and
$\Phi$ from $E$ to $F$ such that $\mathbf{P}[(X_n,Y_n)\in 
A|X_{n-1},Y_{n-1}]=\int \mathbf{1}_A(x,y)\,P(X_{n-1},dx)\,\Phi(x,dy)$.
\end{enumerate}
Such a process is called a stationary \emph{hidden Markov model}; its 
dependence structure is illustrated schematically in Figure \ref{fig:hmm}. 
In applications, $(X_k)_{k\in\mathbb{Z}}$ represents a ``hidden'' process 
which is not directly observable, while the observable process 
$(Y_k)_{k\in\mathbb{Z}}$ represents ``noisy observations'' of the hidden 
process \cite{CMR05}.

\begin{figure}[t]
{\normalsize
\begin{equation*}
\begin{CD}
\cdots @>P>> X_{n-1} @>P>> X_n @>P>> X_{n+1} @>P>> \cdots \\
@. @V{\Phi}VV @V{\Phi}VV @V{\Phi}VV @. \\
\cdots @. Y_{n-1} @. Y_n @. Y_{n+1} @. \cdots
\end{CD}
\end{equation*}}
\caption{Dependence structure of a hidden Markov model.\label{fig:hmm}}
\end{figure}

Of fundamental importance in the theory of hidden Markov models is the 
\emph{nonlinear filter} $(\pi_k)_{k\ge 0}$, defined as the regular 
conditional probability
$$
	\pi_n = \mathbf{P}[X_n\in\,\cdot\,|Y_1,\ldots,Y_n]. 
$$ 
That is, $\pi_n$ is the conditional distribution of the current state of 
the hidden process given the observations to date.  It is a basic fact in 
this theory that the filtering process $(\pi_k)_{k\ge 0}$ is itself a 
Markov process taking values in the space $\mathcal{P}(E)$ of probability 
measures on $E$, whose transition kernel $\mathsf{\Pi}$ can be expressed 
in terms of the transition kernels $P$ and $\Phi$ that determine the model 
(this and other basic facts on nonlinear filters are reviewed in the 
appendix).

Following Kunita \cite{Kun71}, we will be interested in the structure of 
the space of $\mathsf{\Pi}$-invariant probability measures in 
$\mathcal{P}(\mathcal{P}(E))$.  It is easily seen that for every 
$\mathsf{\Pi}$-invariant measure 
$\mathsf{m}\in\mathcal{P}(\mathcal{P}(E))$, the barycenter 
$\mu\in\mathcal{P}(E)$ of $\mathsf{m}$ must be invariant for the 
transition kernel $P$ of the hidden process.  Conversely, for every 
$P$-invariant measure $\mu\in\mathcal{P}(E)$, there exists at least one 
$\mathsf{\Pi}$-invariant measure 
$\mathsf{m}\in\mathcal{P}(\mathcal{P}(E))$ whose barycenter is $\mu$. 
However, the latter need not be unique.

\begin{thm}[Kunita]
\label{thm:exchg}
Let $\mathbf{P}[X_0\in\,\cdot\,]:=\mu$ be the $P$-invariant measure 
defined by the stationary hidden Markov model $(X_k,Y_k)_{k\in\mathbb{Z}}$ 
as above.  If
\begin{equation}
\label{eq:exchg}
	\bigcap_{n\le 0}\big(\mathcal{F}_{-\infty,0}^Y\vee
	\mathcal{F}_{-\infty,n}^X\big) =
	\mathcal{F}_{-\infty,0}^Y\quad\mathbf{P}\mbox{-a.s.},
\end{equation}
then there exists a unique $\mathsf{\Pi}$-invariant measure with 
barycenter $\mu$.  The converse holds if in addition $\Phi$ possesses a 
transition density with respect to some $\sigma$-finite reference measure.
$[$Here $\mathcal{F}_{-\infty,0}^Y:=\sigma\{Y_k:k\le 0\}$,
$\mathcal{F}_{-\infty,n}^X:=\sigma\{X_k:k\le n\}$.$]$
\end{thm}

\begin{rem}
Though the main ideas of the proof are implicitly contained in \cite{Kun71},
this simple and general statement does not appear in the literature 
without various additional simplifying assumptions.  For 
completeness, and in order to make this paper self-contained, we therefore 
include the proof in the appendix.
\end{rem}

Theorem \ref{thm:exchg} is not actually stated as such by Kunita 
\cite{Kun71}.  Instead, Kunita assumes that the hidden process 
$(X_k)_{k\in\mathbb{Z}}$ is purely nondeterministic:

\begin{defn}
A stochastic process $(X_k)_{k\in\mathbb{Z}}$ is called \emph{purely 
nondeterministic} if its past tail $\sigma$-field
$\bigcap_{n\le 0}\mathcal{F}_{-\infty,n}^X$ is $\mathbf{P}$-a.s.\ trivial.
\end{defn}

Kunita's main theorem states\footnote{
	In fact, Kunita's paper is written in the context of a
	continuous time model with white noise observations.
	None of these specific features are used in the proofs, however.
} that if the hidden process 
$(X_k)_{k\in\mathbb{Z}}$ is purely nondeterministic, then there
exists a unique $\mathsf{\Pi}$-invariant measure with barycenter $\mu$.
Kunita's proof, however, does not establish this claim.  Indeed, at the 
crucial point in the proof (\cite{Kun71}, top of p.\ 384), Kunita 
implicitly takes for granted that the following exchange of intersection 
and supremum is permitted:
\begin{equation}
\label{eq:incorr}
	\bigcap_{n\le 0}\big(\mathcal{F}_{-\infty,0}^Y\vee
	\mathcal{F}_{-\infty,n}^X\big) \,\stackrel{?}{=}\,
	\mathcal{F}_{-\infty,0}^Y\vee
	\bigcap_{n\le 0}\mathcal{F}_{-\infty,n}^X
	\quad\mathbf{P}\mbox{-a.s.}
\end{equation}
If this exchange were justified, then Kunita's result would indeed follow 
immediately from Theorem \ref{thm:exchg}.  However, in general, such an 
exchange of intersection and supremum is \emph{not} permitted, as will be
shown in section \ref{sec:simplex} below.

The goal of this paper is to settle, in the negative, a natural
conjecture on the validity of the identity (\ref{eq:incorr}).  Before we 
can describe the conjecture, we must review what is known about the 
validity of (\ref{eq:incorr}) in the filtering setting.

\begin{rem}
Beyond the relevance of (\ref{eq:incorr}) to filtering theory, the 
problem studied in this paper provides a case study on an enigmatic 
problem: when is the exchange of countable intersection and supremum 
of $\sigma$-fields permitted?  Such problems arise in remarkably 
diverse areas of probability theory.  The following references provide 
some further context on this general problem.
\begin{enumerate}
\item Several distinguished mathematicians have given erroneous proofs
related to the exchange of intersection and supremum of $\sigma$-fields, 
including Kolmogorov (see \cite{Sin89}, p.\ 837) and Wiener 
(see \cite{Mas66}, pp.\ 91--93).
\item A simple counterexample to the validity of the exchange of 
intersection and supremum due to Barlow and Perkins can be found in
\cite{Wil91}, p.\ 48.  This example is closely related
to the example given in section \ref{sec:simplex} below.  See also
\cite{CY03}, pp.\ 29--30 and the references therein.
\item The exchange of intersection and supremum appears in diverse
probabilistic settings: see \cite{Weiz83}, section 5 
and the references therein for various examples and counterexamples.
In particular, the innovations problem and several variants of 
Tsirelson's celebrated counterexample provide a rich setting in which
one can study the exchange of intersection and supremum problem; see
\cite{YY11,ES01,Lau10,BL07} and the references therein. See also
\cite{Van10sic} for a different connection to filtering theory.
\item Von Weizs\"acker \cite{Weiz83} gives a general necessary and 
sufficient condition for validity of the exchange of intersection and 
supremum, which is however often difficult to apply in practice.  It is 
shown in \cite{CLP07} that the exchange of intersection and supremum is 
always valid in a given probability space if and only if its probability 
measure is purely atomic.
\end{enumerate}
\end{rem}

\subsection{A simple counterexample}
\label{sec:simplex}

The gap in Kunita's proof was discovered in \cite{BCL04}, where a 
simple counterexample to (\ref{eq:incorr}) was given.  The following
variant of this example will be helpful in understanding our main result.

Let $(\xi_k)_{k\in\mathbb{Z}}$ be an i.i.d.\ sequence of (Bernoulli) 
random variables uniformly distributed in $\{0,1\}$.  Let 
$E=\{0,1\}\times\{0,1\}$ and $F=\{0,1\}$, and define the stochastic 
process $(X_k,Y_k)_{k\in\mathbb{Z}}$ taking values in $E\times F$ as 
follows:
$$
	X_n = (\xi_{n-1},\xi_n),\qquad\qquad
	Y_n = |\xi_n-\xi_{n-1}|.
$$
It is evident that $(X_k,Y_k)_{k\in\mathbb{Z}}$ is a stationary hidden 
Markov model.  Note that:
\begin{itemize}
\item Clearly $\xi_0 = (\xi_{n-1} + Y_n + \cdots + Y_0)
\mathop{\mathrm{mod}}2$ for any $n\le 0$.  Therefore,
$$
	\xi_0 \quad\mbox{is}\quad
	\bigcap_{n\le 0}\big(\mathcal{F}_{-\infty,0}^Y\vee
	\mathcal{F}_{-\infty,n}^X\big)\mbox{-measurable}.
$$
\item On the other hand, as
$\mathbf{P}[\xi_0=0|\mathcal{F}_{-\infty,0}^Y]=1/2$
by direct computation,
$$
	\xi_0 \quad\mbox{is not}\quad
	\mathcal{F}_{-\infty,0}^Y\mbox{-measurable }	
	\mathbf{P}\mbox{-a.s.}
$$
\item $(X_k)_{k\in\mathbb{Z}}$ is purely 
nondeterministic by the Kolmogorov zero-one law.
\end{itemize}
Therefore, evidently the identity (\ref{eq:incorr}) does not hold in 
this example.

\subsection{A positive result and a conjecture}

In view of the counterexample above, one might expect that the gap in 
Kunita's proof cannot be resolved in general.  However, it turns out that 
such counterexamples are extremely fragile.  For example, let 
$(\gamma_k)_{k\in\mathbb{Z}}$ be an i.i.d.\ sequence of standard Gaussian 
random variables, and let us modify the observation model in the above 
example to
$$
	Y_n = |\xi_n-\xi_{n-1}| + \varepsilon\gamma_n.
$$
Then it can be verified that for arbitrarily small $\varepsilon>0$, the 
identity (\ref{eq:incorr}) holds.  It is only in the degenerate case 
$\varepsilon=0$ that (\ref{eq:incorr}) fails.  This suggests that the 
presence of some amount of noise, however small, is sufficient in order to 
ensure the validity of (\ref{eq:incorr}).  This intuition can be made 
precise in a surprisingly general setting, which is established by the 
following result due to the author \cite{vH09}.  Here the notion of 
nondegeneracy formalizes the presence of observation noise.

\begin{defn}
The hidden Markov model $(X_k,Y_k)_{k\in\mathbb{Z}}$ is said to possess
\emph{nondegenerate observations} if there exist a $\sigma$-finite 
reference measure $\varphi$ on $F$ and a strictly positive measurable 
function $g:E\times F\to\mbox{}]0,\infty[\mbox{}$ such that
$$
	\Phi(x,A) = \int \mathbf{1}_A(y)\,g(x,y)\,\varphi(dy)
	\quad\mbox{for all }x\in E,~A\in\mathcal{B}(F).
$$
\end{defn}

\begin{thm}[\cite{vH09}]
\label{thm:vh}
Given a stationary hidden Markov model $(X_k,Y_k)_{k\in\mathbb{Z}}$ 
as defined in this section, with $P$-invariant measure 
$\mathbf{P}[X_0\in\,\cdot\,]:=\mu$, assume that:
\begin{enumerate}
\item The hidden process $(X_k)_{k\in\mathbb{Z}}$ is absolutely regular:
\begin{equation}
\label{eq:harris}
	\mathbf{E}\big[\|\mathbf{P}[X_n\in\,\cdot\,|X_0]-\mu
	\|_{\rm TV}\big]
	\xrightarrow{n\to\infty}0.
\end{equation}
\item The observations are nondegenerate. 
\end{enumerate}
Then the identity (\ref{eq:exchg}) holds true.
\end{thm}

This result resolves the validity of (\ref{eq:exchg}) in many cases of 
interest.  Indeed, the mixing assumption (\ref{eq:harris}) holds in a very 
broad class of applications, and a well-established theory provides a 
powerful set of tools to verify this assumption \cite{MT09}.  
Nonetheless, the assumption (\ref{eq:harris}) is strictly stronger than 
the assumption that the hidden process is purely nondeterministic;
the latter is equivalent to
$$
	\mathbf{E}\big[|\mathbf{P}[X_n\in A|X_0]-\mu(A)
	|\big]
	\xrightarrow{n\to\infty}0\quad
	\mbox{for all }A\in\mathcal{B}(E)
$$
(see \cite{Tot70}, Proposition 3).  If, as one might conjecture, 
nondegeneracy of the observations suffices to justify the exchange 
of intersection and supremum (\ref{eq:incorr}), then Theorem \ref{thm:vh} 
should already hold when the hidden process is only purely 
nondeterministic, i.e., Kunita's claim would hold true whenever the 
observations are nondegenerate.  This stronger result was conjectured in 
\cite{vH09}, pp.\ 1877--1878.

\begin{conj}
\label{conj}
If the hidden process is purely nondeterministic and the observations are 
nondegenerate, then (\ref{eq:exchg}) holds true.
\end{conj}

Conjecture \ref{conj} seems tantalizingly close to Theorem \ref{thm:vh}, 
particularly if we rephrase (\ref{eq:harris}) in terms of tail 
$\sigma$-fields.  Indeed, let $\mathbf{P}^x$ be a version of the regular
conditional probability $\mathbf{P}^{X_0}=\mathbf{P}[\,\cdot\,|X_0]$.  
Then, from the results of \cite{vH09}, for example, one may read off the 
following equivalent formulation of (\ref{eq:harris}):
\begin{align*}
	&\mbox{There exists a set }E_0\in\mathcal{B}(E)
	\mbox{ such that }\mu(E_0)=1\mbox{ and}\\
	&
	\mbox{for all }A\in
	\textstyle{\bigcap_{n\le 0}\mathcal{F}_{-\infty,n}^X}
	\mbox{ and }x,y\in E_0
	~~~
	\mathbf{P}^x[A]=\mathbf{P}^y[A]\in\{0,1\}.
\end{align*}
On the other hand, clearly $(X_k)_{k\in\mathbb{Z}}$ is purely 
nondeterministic if and only if
\begin{align*}
	&\mbox{For any }A\in
	\textstyle{\bigcap_{n\le 0}\mathcal{F}_{-\infty,n}^X},
	\mbox{ there exists }E_0\in\mathcal{B}(E)
	\mbox{ (depending possibly}\\
	&\mbox{on }A\mbox{) such that }\mu(E_0)=1
	\mbox{ and for all }x,y\in E_0
	~~~
	\mathbf{P}^x[A]=\mathbf{P}^y[A]\in\{0,1\}.
\end{align*}
Thus the difference between the assumptions is that in the latter, the set 
$E_0$ may depend on $A$, while in the former $E_0$ does not depend on $A$.

\subsection{Main result}

The main result of this paper is that Conjecture \ref{conj} is false.
We establish this by exhibiting a counterexample.

\begin{thm}
\label{thm:result}
There exists a stationary hidden Markov model $(X_k,Y_k)_{k\in\mathbb{Z}}$ 
in a Polish state space $E\times F$ such that the hidden process is purely 
nondeterministic and the observations are nondegenerate, but nonetheless
(\ref{eq:exchg}) fails to hold.  

Moreover, this model may be constructed such that the transition kernel 
$P$ of the hidden process is Feller, and such that the observations are of 
standard additive noise type $Y_n = h(X_n)+\varepsilon\gamma_n$ where 
$h:E\to\mathbb{R}^3$ is a bounded continuous function, $\varepsilon>0$ and 
$(\gamma_k)_{k\in\mathbb{Z}}$ are standard Gaussian random variables in 
$\mathbb{R}^3$.
\end{thm}

The counterexample to Conjecture \ref{conj}, whose existence is guaranteed 
by this result, must surely yield a nasty filtering problem!  Yet, Theorem 
\ref{thm:result} indicates the model need not even be \emph{too} nasty: 
the example can be chosen to satisfy standard regularity assumptions and 
using a perfectly ordinary observation model.  It therefore seems doubtful 
that the general result of Theorem \ref{thm:vh} can be substantially 
weakened; absolute regularity (\ref{eq:harris}) is evidently essential.

Let us briefly explain the intuition behind the counterexample.  We aim to 
mimic the noiseless counterexample in section \ref{sec:simplex}.  The 
idea is to construct a variant of that model which has very long memory: 
we can then hope to average out the additional observation noise (needed 
to make the observations nondegenerate), reverting essentially to the 
noiseless case.  On the other hand, we cannot give the process such long 
memory that it ceases to be purely nondeterministic.  The following 
construction strikes a balance between these competing goals.  We 
reconsider the example of section \ref{sec:simplex} not as a time series, 
but as a random scenery.  We then construct a stochastic process by 
running a random walk on the integers, and reporting at each time the 
value of the scenery at the current location of the walk.  The resulting 
\emph{random walk in random scenery} \cite{HS06,Kes98} is purely 
nondeterministic, yet has a very long memory due to the recurrence of the 
random walk.  The latter is exploited by a remarkable \emph{scenery 
reconstruction} result of Matzinger and Rolles \cite{MR03} which allows us 
to average out the observation noise.  Theorem \ref{thm:result} follows 
essentially by combining the scenery reconstruction with the example of 
section \ref{sec:simplex}, except that we must work in a slightly larger 
state space for technical reasons.

\begin{rem}
Random walks in random scenery are closely related to the 
$T,T^{-1}$-process, which was conjectured by Weiss (\cite{Wei72}, p.\ 682) 
and later proved by Kalikow \cite{Kal82} to be a natural example of a 
$K$-process that is not a $B$-process.  In the language of ergodic theory, 
a purely nondeterministic process is a $K$-process \cite{Orn73} while a 
process that satisfies (\ref{eq:exchg}) is an
$\mathcal{F}_{-\infty,0}^Y$-relative $K$-process \cite{Rah78}.  Our example 
may thus be interpreted as a $K$-process that is not $K$ relative to a 
nondegenerate observation process.  The absolute regularity property
(\ref{eq:harris}) is equivalent to the weak Bernoulli property in ergodic
theory (cf.\ \cite{VR59}).
\end{rem}

We end this section with a brief discussion of the practical implications 
of Theorem \ref{thm:result}.  The mixing assumption (\ref{eq:harris}) 
required by Theorem \ref{thm:vh} states that the law of the hidden process 
converges in the sense of total variation to the invariant measure $\mu$ 
for almost every initial condition.  This occurs in a wide variety of 
applications \cite{MT09}, as long as the hidden state space $E$ is finite 
dimensional.  In infinite dimensions, however, most probability measures 
are mutually singular and total variation convergence is rare.  When the 
hidden process is defined by the solution of a stochastic partial 
differential equation, for example, typically the best we can hope for is 
weak convergence to the invariant measure.  In this case (\ref{eq:harris}) 
fails, though the process is still purely nondeterministic.  Our main 
result indicates that nice ergodic properties of the nonlinear filter 
cannot be taken for granted in the infinite dimensional setting. This is 
unfortunate, as infinite dimensional filtering problems appear naturally 
in important applications such as weather prediction and geophysical or 
oceanographic data assimilation (see, e.g., \cite{Maj10}), while 
ergodicity of the nonlinear filter is essential to reliable performance 
of filtering algorithms \cite{vHipf}.  The current state of knowledge on 
the ergodic theory of infinite dimensional filtering problems appears to 
be essentially nonexistent.

The remainder of this paper is organized as follows.  In section 
\ref{sec:construct} we introduce the various stochastic processes needed 
to construct our counterexample.  Sections \ref{sec:pf} and \ref{sec:pf2} 
are devoted to the proof of Theorem \ref{thm:result}.  The appendix 
reviews the ergodic theory of nonlinear filters (including a proof of 
Theorem \ref{thm:exchg}).

\section{Construction}
\label{sec:construct}

In the following, we will work on the canonical probability space
$(\Omega,\mathcal{F},\mathbf{P})$ which supports the following 
independent random variables.
\begin{itemize}
\item $(\eta_k)_{k\in\mathbb{Z}}$, $\xi_0$ are i.i.d.\ random 
variables, uniformly distributed in $\{0,1,2\}$.
\item $(\delta_k)_{k\in\mathbb{Z}}$ are i.i.d.\ random variables,
uniformly distributed in $\{-1,1\}$.
\item $(\gamma_k)_{k\in\mathbb{Z}}$ are i.i.d.\ standard Gaussian random 
variables in $\mathbb{R}^3$.
\end{itemize}
Denote by $\{e(0),e(1),e(2)\}\subset\mathbb{R}^3$ the canonical basis in 
$\mathbb{R}^3$.

We now proceed to define various stochastic processes.  Define recursively
\begin{equation*}
	\xi_n = \left\{
	\begin{array}{ll}
	(\xi_{n-1} + \eta_n) \mathop{\mathrm{mod}} 3 &\quad\mbox{for }n>0,\\
	(\xi_{n+1} - \eta_{n+1}) \mathop{\mathrm{mod}} 3 &\quad\mbox{for }n<0.\\
	\end{array}\right.
\end{equation*}
Note that $(\xi_k)_{k\in\mathbb{Z}}$ is an i.i.d.\ sequence uniformly 
distributed in $\{0,1,2\}$, and
$$
	\eta_n = (\xi_n - \xi_{n-1}) \mathop{\mathrm{mod}} 3.
$$
Next, we define the simple random walk $(N_k)_{k\in\mathbb{Z}}$ on 
$\mathbb{Z}$ as
\begin{equation*}
	N_n = \left\{
	\begin{array}{ll}
	\sum_{k=1}^n\delta_k &\quad\mbox{for }n\ge 0,\\
	-\sum_{k=n+1}^{0}\delta_k &\quad\mbox{for }n<0.
	\end{array}
	\right.
\end{equation*}
We can now define the random walk in random scenery 
$(Z_k)_{k\in\mathbb{Z}}$ which takes values in the set
$\{-1,1\}\times\{0,1,2\}\times\{0,1,2\}:= I$ as follows:
$$
	Z_n = (Z_{n,0},Z_{n,1},Z_{n,2}) =
	(\delta_{n+1},\xi_{N_n - 1},\xi_{N_n}).
$$
It is not difficult to see that $(Z_n)_{n\in\mathbb{Z}}$ is a 
stationary process.  We finally make the process Markovian by defining
the $I^{\mathbb{Z}_+}$-valued process $(X_n)_{n\in\mathbb{Z}}$ as
$$
	X_n = (Z_k)_{k\ge n}\qquad
	\mbox{(that is, }X_{n,k} = Z_{n+k}\mbox{ for }k\in\mathbb{Z}_+
	\mbox{)},
$$
and we define the $\mathbb{R}^3$-valued observation process 
$(Y_k)_{k\in\mathbb{Z}}$ as
$$
	Y_n = h(X_n)+\varepsilon\gamma_n =
	e(\eta_{N_n})+\varepsilon\gamma_n,
$$
where $\varepsilon>0$ is a fixed constant and
$h:I^{\mathbb{Z}_+}\to\mathbb{R}^3$ is defined as
$$
	h(x)=e((x_{0,2}-x_{0,1})\mathop{\mathrm{mod}} 3).
$$
It is evident that the pair $(X_n,Y_n)_{n\in\mathbb{Z}}$ defines a 
stationary hidden Markov model taking values in the Polish space
$I^{\mathbb{Z}_+}\times\mathbb{R}^3$ and with nondegenerate observations.

Let us define the $\sigma$-fields
$$
	\mathcal{F}_{m,n}^X = \sigma\{X_k:k\in[m,n]\},\qquad\quad
	\mathcal{F}_{m,n}^Y = \sigma\{Y_k:k\in[m,n]\},
$$
for $m,n\in\mathbb{Z}$, $m\le n$.  The $\sigma$-fields 
$\mathcal{F}_{-\infty,n}^X$, $\mathcal{F}_{m,\infty}^X$, etc., are defined 
in the usual fashion (for example, $\mathcal{F}_{-\infty,n}^X=
\bigvee_{m\le n}\mathcal{F}_{m,n}^X$).  Our main result is now as follows.

\begin{thm}
\label{thm:mainres}
For the hidden Markov model $(X_k,Y_k)_{k\in\mathbb{Z}}$ with 
nondegenerate observations, as defined in this section, the following
hold:
\begin{enumerate}
\item The future tail $\sigma$-field 
$$
	\mathcal{T}:=\bigcap_{n\ge 0}\mathcal{F}^X_{n,\infty}
	\quad\mbox{is }\mathbf{P}\mbox{-a.s.\ trivial}.
$$
\item We have the strict inclusion
$$
	\bigcap_{n\ge 0}(\mathcal{F}^Y_{0,\infty}\vee\mathcal{F}^X_{n,\infty})
	\supsetneq \mathcal{F}^Y_{0,\infty}
	\quad\mathbf{P}\mbox{-a.s.},
$$
provided that $\varepsilon>0$ is chosen sufficiently small.
\end{enumerate}
\end{thm}

The proof of this result, given in section \ref{sec:pf} below, is based on 
mixing and reconstruction results for random walks in random scenery
\cite{Mei74,MR03}.

The model of Theorem \ref{thm:mainres} is time-reversed from the 
counterexample to be provided by Theorem \ref{thm:result}.  It is 
immediate from the Markov property, however, that the time reversal 
of a stationary hidden Markov model yields again a stationary hidden 
Markov model.  Therefore, the following corollary is immediate:

\begin{cor}
\label{cor:preresult}
For $\varepsilon>0$ sufficiently small, the time-reversed model 
$$
	(\tilde X_k,\tilde 
	Y_k)_{k\in\mathbb{Z}}:=(X_{-k},Y_{-k})_{k\in\mathbb{Z}}
$$
is purely nondeterministic and has nondegenerate observations,
but (\ref{eq:exchg}) fails.
\end{cor}

This proves the first part of Theorem \ref{thm:result} and settles 
Conjecture \ref{conj}.  However, when constructed in this manner, the 
transition kernel of $(\tilde X_k)_{k\in\mathbb{Z}}$ cannot be chosen to 
satisfy the Feller property on $I^{\mathbb{Z}_+}$.  Some further effort is 
therefore required to complete the proof of Theorem \ref{thm:result}, 
which we postpone to section \ref{sec:pf2}.

\section{Proof of Theorem \ref{thm:mainres}}
\label{sec:pf}

\subsection{First part}

Consider the stochastic process $\tilde\xi_n:=(\xi_{n-1},\xi_n)$. It is 
easily seen that this is a stationary, irreducible and aperiodic Markov 
chain taking values in the space $\{0,1,2\}\times\{0,1,2\}$, so that 
$(\tilde\xi_k)_{k\in\mathbb{Z}}$ is an ergodic process.  The triviality of 
$\mathcal{T}$ now follows from the Theorem in \cite{Mei74}, p.\ 267 (this 
follows in particular from equation (3) in \cite{Mei74} using 
\cite{Smo71}, Theorem 7.9).

\subsection{Second part}

Consider the modified observation process $(Y_k')_{k\in\mathbb{Z}}$
taking values in $\{0,1,2\}$, defined as follows:
$$
	Y_n' = \mathop{\mathrm{argmax}}_{i=0,1,2} Y_{n,i}.
$$
That is, $Y_n'$ is the coordinate index of the largest component of
the vector $Y_n\in\mathbb{R}^3$.  By symmetry, it is easily seen that
for some $\delta>0$ depending on $\varepsilon$
$$
	\mathbf{P}[Y_n'=i|\eta_{N_n}=j] = 
	\frac{\delta}{3}\quad\forall\,i\ne j,
	\qquad
	\mathbf{P}[Y_n'=i|\eta_{N_n}=i] = 1-\frac{2\delta}{3}
	\quad\forall i,
$$
where $\delta\downarrow 0$ as $\varepsilon\downarrow 0$.  The conditional 
law of $Y_n'$ can therefore be generated as follows: draw a Bernoulli 
random variable with parameter $\delta$; if it is zero, set 
$Y_n'=\eta_{N_n}$, otherwise let $Y_n'$ be a random draw from the uniform 
distribution on $\{0,1,2\}$.  We can now apply the 
scenery reconstruction result from \cite{MR03}.

\begin{defn}
Let $x,y\in\{0,1,2\}^{\mathbb{Z}}$.  We write $x\approx y$ if
there exist $a\in\{-1,1\}$ and $b\in\mathbb{Z}$ such that
$x_n=y_{an+b}$ for all $n\in\mathbb{Z}$ (that is, $x\approx y$ iff
the sequences $x$ and $y$ agree up to translation and/or reflection).
\end{defn}

\begin{thm}[\cite{MR03}] \label{thm:mr}
There is a measurable map 
$\iota:\{0,1,2\}^{\mathbb{Z}_+}\to\{0,1,2\}^{\mathbb{Z}}$
such that $\mathbf{P}[\iota((Y_k')_{k\ge 0})\approx 
(\eta_k)_{k\in\mathbb{Z}}]=1$ provided $\varepsilon>0$ is sufficiently 
small.
\end{thm}

From now on, let us fix $\varepsilon>0$ sufficiently small and the map 
$\iota$ as in Theorem \ref{thm:mr}.  By the definition of the equivalence 
relation $\approx$, there exist
$\mathcal{F}_{0,\infty}^Y\vee\mathcal{F}_{-\infty,\infty}^{\eta}$-measurable 
random variables $A$ and $B$, taking values in $\{-1,1\}$ and 
$\mathbb{Z}$, respectively, such that $\iota((Y_k')_{k\ge 0})_n=
\eta_{An+B}$ $\mathbf{P}$-a.s.\ for all $n\in\mathbb{Z}$.

\begin{rem}
Let us note that, even though by construction 
$(\eta_{Ak+B})_{k\in\mathbb{Z}}$ is a.s.\ 
$\mathcal{F}_{0,\infty}^Y$-measurable, it is not possible for the random 
variables $A$ and $B$ to be $\mathcal{F}_{0,\infty}^Y$-measurable; see 
\cite{Kes98}, Remark (ii).  This will not be a problem for us.
\end{rem}

The point of the above construction is the following claim: the random 
variable $\xi_B$ is a.s.\ $\bigcap_n (\mathcal{F}_{0,\infty}^Y\vee
\mathcal{F}_{n,\infty}^X)$-measurable, but it is not a.s.\ 
$\mathcal{F}_{0,\infty}^Y$-measurable.  This clearly suffices to prove 
the result.  It thus remains to establish the claim.

\begin{lem}
The random variable $\xi_B$ is $\mathbf{P}$-a.s.\ 
$\bigcap_n (\mathcal{F}_{0,\infty}^Y\vee
\mathcal{F}_{n,\infty}^X)$-measurable.
\end{lem}

\begin{proof}
Fix $n\in\mathbb{Z}$.  Define the random variables 
$(\tau_k)_{k\in\mathbb{Z}}$ as
$$
	\tau_j = \inf\left\{k\ge 0:
	\sum_{i=0}^{k-1} X_{n,i,0}=j
	\right\},
$$
and define the random variables $(\xi_k')_{k\in\mathbb{Z}}$ as
$$
	\xi_j' = X_{n,\tau_j,2}\,\mathbf{1}_{\tau_j<\infty}.
$$
Then clearly $(\xi_k')_{k\in\mathbb{Z}}$ is 
$\mathcal{F}_{n,\infty}^X$-measurable and 
$\mathbf{P}[(\xi_k')_{k\in\mathbb{Z}}\approx 
(\xi_k)_{k\in\mathbb{Z}}]=1$.

We now claim that we can ``align'' $(\xi_k')_{k\in\mathbb{Z}}$ with
$(\eta_{Ak+B})_{k\in\mathbb{Z}}$.  Indeed, note that for any 
$b\in\mathbb{Z}$, we can estimate
$$
	\mathbf{P}\left[
	\eta_{k} = \eta_{k+b} \mbox{ for all }k\in\mathbb{Z}
	\right] \le
	\mathbf{P}\left[
	\eta_{0} = \eta_{kb} \mbox{ for all }k\ge 1
	\right] = 0,
$$
$$
	\mathbf{P}\left[
	\eta_{k} = \eta_{-k+b} \mbox{ for all }k\in\mathbb{Z}
	\right] \le
	\prod_{k=b}^\infty
	\mathbf{P}\left[
	\eta_{k} = \eta_{-k+b}
	\right] = 0,
$$
where we have used that $(\eta_k)_{k\in\mathbb{Z}}$ are i.i.d.\ and 
nondeterministic.  Therefore
$$
	\mathbf{P}\left[
	\mbox{there exist }a\in\{-1,1\},~b\in\mathbb{Z}\mbox{ such
	that }
	\eta_{k} = \eta_{ak+b} \mbox{ for all }k\in\mathbb{Z}
	\right] = 0.
$$
In particular, if we define $(\eta_k')_{k\in\mathbb{Z}}$ as
$$
	\eta_j' = (\xi_j' - \xi_{j-1}') \mathop{\mathrm{mod}} 3,
$$
it follows that there must exist $\mathbf{P}$-a.s.\ unique 
$\mathcal{F}_{0,\infty}^Y\vee\mathcal{F}_{n,\infty}^X$-measurable
random variables $A'$ and $B'$, taking values in $\{-1,1\}$ and
$\mathbb{Z}$, respectively, such that
$$
	\eta_{A'j+B'}' = \eta_{Aj+B}\quad
	\mbox{for all }j\in\mathbb{Z}\quad
	\mathbf{P}\mbox{-a.s.}
$$
It follows by uniqueness that
$$
	\xi_{A'j+B'}' = \xi_{Aj+B}\quad
	\mbox{for all }j\in\mathbb{Z}\quad
	\mathbf{P}\mbox{-a.s.}
$$
In particular, $\xi_{B'}'=\xi_B$ $\mathbf{P}$-a.s.  But $\xi_{B'}'$ is
$\mathcal{F}_{0,\infty}^Y\vee\mathcal{F}_{n,\infty}^X$-measurable
by construction.  Therefore, we have shown that
$\xi_{B}$ is $\mathbf{P}$-a.s.\ 
$\mathcal{F}_{0,\infty}^Y\vee\mathcal{F}_{n,\infty}^X$-measurable.
As the choice of $n$ was arbitrary, the proof is easily completed.
\end{proof}

\begin{lem}
The random variable $\xi_B$ is not $\mathbf{P}$-a.s.\ 
$\mathcal{F}_{0,\infty}^Y$-measurable.
\end{lem}

\begin{proof}
Note that $\mathbf{P}$-a.s.
\begin{equation*}
\begin{split}
	&\mathbf{P}[\xi_B=i,~B=j|\mathcal{F}_{-\infty,\infty}^\eta\vee
	\mathcal{F}_{-\infty,\infty}^\delta\vee
	\mathcal{F}_{-\infty,\infty}^\gamma] \\
	&\qquad\mbox{} =
	\mathbf{1}_{B=j} \,
	\mathbf{P}[\xi_j=i|\mathcal{F}_{-\infty,\infty}^\eta\vee
	\mathcal{F}_{-\infty,\infty}^\delta\vee
	\mathcal{F}_{-\infty,\infty}^\gamma]
	\\
	&\qquad\mbox{} =
	\mathbf{1}_{B=j} \,\{
	\mathbf{P}[\xi_0=i|\mathcal{F}_{-\infty,\infty}^\eta\vee
	\mathcal{F}_{-\infty,\infty}^\delta\vee
	\mathcal{F}_{-\infty,\infty}^\gamma]
	\circ\Theta^j\}
	\\
	&\qquad\mbox{} =
	\mathbf{1}_{B=j}\,
	\mathbf{P}[\xi_0=i].
\end{split}
\end{equation*}
Here we have used that $B$ is 
$\mathcal{F}_{0,\infty}^Y\vee\mathcal{F}_{-\infty,\infty}^{\eta}$-measurable
for the first equality, stationarity of the law of 
$(\xi_k,\eta_k,\delta_k,\gamma_k)_{k\in\mathbb{Z}}$ for the second 
equality ($\Theta$ denotes the canonical shift), and independence of
$\xi_0$ and $(\eta_k,\delta_k,\gamma_k)_{k\in\mathbb{Z}}$ for the third
equality.  Summing over $j$, and conditioning on 
$\mathcal{F}_{0,\infty}^Y$, we obtain
$$
	\mathbf{P}[\xi_B=i|\mathcal{F}_{0,\infty}^Y]=
	\mathbf{P}[\xi_0=i]=1/3\quad\mathbf{P}\mbox{-a.s.}
$$
Thus $\xi_B$ is independent from $\mathcal{F}_{0,\infty}^Y$,
hence not $\mathbf{P}$-a.s.\ $\mathcal{F}_{0,\infty}^Y$-measurable.
\end{proof}

\begin{rem}
The additive noise model $Y_n = h(X_n)+\varepsilon\gamma_n$ is 
inessential to the proof; we could have just as easily started 
from the $\{0,1,2\}$-valued observation model $Y_n'$ as in 
\cite{MR03}.  The only reason we have chosen to construct our example with 
the additive noise model is to make the point that there is nothing 
special about the choice of observations: one does not have to ``cook up'' 
a complicated observation model to make the counterexample work.  All the 
unpleasantness arises from the ergodic theory of random walks in 
random scenery.
\end{rem}

\section{Proof of Theorem \ref{thm:result}}
\label{sec:pf2}

For any $x\in I^{\mathbb{Z}_+}$, define
$$
	\tau_j(x) = \inf\left\{
	k\ge 0:\sum_{i=0}^{k-1} x_{i,0}=j
	\right\}.
$$
Now define the space
$$
	E := \left\{
	x\in I^{\mathbb{Z}_+}:
	\tau_j(x)<\infty\mbox{ for all }j\in\mathbb{Z}
	\right\}\subset I^{\mathbb{Z}_+}.
$$
We endow $E$ with the topology of pointwise convergence (inherited from
$I^{\mathbb{Z}_+}$).

\begin{lem}
\label{lem:polish}
$E$ is Polish.
\end{lem}

\begin{proof}
For $x,x'\in E$, define the metric
$$
	d(x,x') := 
	\sum_{k=0}^\infty 2^{-k}\,\mathbf{1}_{x_k\ne x_k'} +
	\sum_{j=-\infty}^\infty 2^{-|j|}\,\{|\tau_j(x)-\tau_j(x')|
	\wedge 1\}.
$$
It suffices to prove that $d$ metrizes the topology of pointwise 
convergence in $E$ (which is certainly separable) and that $(E,d)$ is a 
complete metric space.

We first prove that $d$ metrizes the topology of pointwise convergence.
Clearly $d(x_n,x)\to 0$ as $n\to\infty$ implies that $x_n\to x$ pointwise.  
Conversely, suppose that $x_n\to x$ as $n\to\infty$ pointwise.  It 
suffices to show that $\tau_j(x_n)\to\tau_j(x)$ as $n\to\infty$ for all 
$j\in\mathbb{Z}$.  But as $\tau_j(x)<\infty$ by assumption (as $x\in E$), 
it follows that $\tau_j(x_n)=\tau_j(x)$ whenever $x_{n,k}=x_k$ for all
$k\le\tau_j(x)$, which is the case for $n$ sufficiently large by pointwise 
convergence.  This establishes the claim.

It remains to show that $(E,d)$ is complete.  To this end, let 
$(x_n)_{n\in\mathbb{N}}$ be a Cauchy sequence for the metric $d$.
Then it is clearly Cauchy for
$$
	\tilde d(x,x') := 
	\sum_{k=0}^\infty 2^{-k}\,\mathbf{1}_{x_k\ne x_k'},
$$
which defines a complete metric for the topology of pointwise 
convergence on $I^{\mathbb{Z}_+}\supset E$.  Therefore, there exists
$x\in I^{\mathbb{Z}_+}$ such that $x_n\to x$ as $n\to\infty$ pointwise.  
It suffices to show that $x\in E$.  Indeed, when this is the case, it 
follows immediately that $d(x_n,x)\to 0$ as $n\to\infty$ (as we have shown 
that $d$ metrizes pointwise convergence in $E$), thus proving completeness 
of $(E,d)$.

To complete the proof, suppose that $x\not\in E$.  Then there exists
$j\in\mathbb{Z}$ such that $\tau_j(x)=\infty$.  In particular, if
$x_{n,k}=x_k$ for all $k\le N<\infty$, then $\tau_j(x_n)>N$.  As this is
the case for $n$ sufficiently large by pointwise convergence, it follows 
that
$$
	\sup_{m\ge n}d(x_m,x_n) \ge
	2^{-|j|}\sup_{m\ge n}|\tau_j(x_m)-\tau_j(x_n)|\wedge 1 =
	2^{-|j|}\quad\mbox{for all }n\ge 1.
$$
This contradicts the Cauchy property of $(x_n)_{n\in\mathbb{N}}$.
\end{proof}

Denote by $\mathbf{P}[X_0\in\,\cdot\,]:=\mu$ the invariant measure of the 
$I^{\mathbb{Z}_+}$-valued Markov process $(X_k)_{k\in\mathbb{Z}}$ defined 
in section \ref{sec:construct}.  It is clear that $E$ is measurable as a 
subset of $I^{\mathbb{Z}_+}$ and that $\mu(E)=1$.  We are going to 
construct a Feller transition kernel $\tilde P$ from $E$ to $E$ with 
stationary measure $\mu$ (restricted to $E$), such that the corresponding 
stationary $E$-valued Markov process coincides a.s.\ with the stationary 
$I^{\mathbb{Z}_+}$-valued Markov process $(\tilde X_k)_{k\in\mathbb{Z}}$ 
defined in section \ref{sec:construct}. 

\begin{lem}
\label{lem:feller}
Define the transition kernel
$\tilde P:E\times\mathcal{B}(E)\to[0,1]$ as follows:
$$
	\tilde P(x,\{T_{1}(x)\}) = \tilde P(x,\{T_{-1}(x)\})=\frac{1}{2},
$$
where $T_a:E\to E$, $a\in\{-1,1\}$ are defined as
$$
	T_a(x) = [(a,x_{\tau_{-a}(x),1},x_{\tau_{-a}(x),2}),x].
$$
Then the law under $\mathbf{P}$ of the process $(\tilde 
X_k)_{k\in\mathbb{Z}}$ defined in section \ref{sec:construct} is that of a 
stationary Markov process taking values in $E$ with transition kernel 
$\tilde P$ and invariant measure $\mu$.  Moreover, $\tilde P$ satisfies 
the Feller property.
\end{lem}

\begin{proof}
It follows along the lines of the proof of Lemma \ref{lem:polish} that the 
functions $T_{1}$ and $T_{-1}$ are continuous.  Therefore, the Feller 
property of $\tilde P$ is immediate.

To complete the proof, it suffices (as clearly $\tilde X_n\in E$ 
$\mathbf{P}$-a.s.\ for all $n\in\mathbb{Z}$ and as
$(\tilde X_k)_{k\in\mathbb{Z}}$ is a stationary Markov process) to show that
$$
	\mathbf{P}[\tilde X_1\in A|\tilde X_0] = \tilde P(\tilde X_0,A)
	\quad\mathbf{P}\mbox{-a.s.}\quad
	\mbox{for all }A\in\mathcal{B}(E).
$$
To this end, note that 
$$
	\tilde X_1 =
	[(\delta_0,\xi_{-\delta_0-1},\xi_{-\delta_0}),\tilde X_0] =
	[(\delta_0,\tilde X_{0,\tau_{-\delta_0}(\tilde X_0),1},
	\tilde X_{0,\tau_{-\delta_0}(\tilde X_0),2}),\tilde X_0]
	\quad\mathbf{P}\mbox{-a.s.}
$$
Moreover, as $\tilde X_0$ is $\mathcal{F}_{-\infty,\infty}^\xi\vee
\mathcal{F}_{1,\infty}^\delta$-measurable, it follows from the 
construction in section \ref{sec:construct} that $\delta_0$ is 
independent of $\tilde X_0$.  The result follows directly.
\end{proof}

\begin{proof}[Proof of Theorem \ref{thm:result}]
Construct the canonical $E\times\mathbb{R}^3$-valued stationary hidden 
Markov model $(X_k',Y_k')_{k\in\mathbb{Z}}$ such that the hidden process 
$(X_k')_{k\in\mathbb{Z}}$ has transition kernel $\tilde P$ and invariant 
measure $X_0'\sim\mu$, and with the observation model 
$Y_n'=h(X_n')+\varepsilon\gamma_n$ where
$(\gamma_k)_{k\in\mathbb{Z}}$ is an i.i.d.\ sequence of standard Gaussian 
random variables in $\mathbb{R}^3$ independent of 
$(X_k')_{k\in\mathbb{Z}}$.  Clearly $E$ and $\mathbb{R}^3$ are Polish by 
Lemma \ref{lem:polish}, the observations are nondegenerate, 
$h:E\to\mathbb{R}^3$ (defined in section \ref{sec:construct}) is bounded 
and continuous, and $\tilde P$ is Feller by Lemma \ref{lem:feller}.  
Moreover, the law of the model $(X_k',Y_k')_{k\in\mathbb{Z}}$ coincides 
with that of $(\tilde X_k,\tilde Y_k)_{k\in\mathbb{Z}}$ as defined in 
section \ref{sec:construct}.  Therefore, by Corollary \ref{cor:preresult}, 
$(X_k')_{k\in\mathbb{Z}}$ is purely nondeterministic but (\ref{eq:exchg}) 
fails for this model when $\varepsilon>0$ is chosen sufficiently small.
\end{proof}

\appendix

\section{Ergodic theory of nonlinear filters}
\label{sec:app}

The goal of the appendix is to collect a few basic results on the ergodic 
theory of nonlinear filters.  Similar results appear in various forms in 
the literature, see, for example, \cite{Bud03,CvH10} and the references 
therein.  However, all known proofs require various simplifying 
assumptions, such as the Feller property or irreducibility of the hidden 
process, nondegenerate observations, etc.  As a general result does not 
appear to be readily available in the literature, we provide here a 
largely self-contained treatment culminating in the proof of Theorem 
\ref{thm:exchg}.

Let us note that analogous results can be obtained for continuous time, 
either by direct arguments (cf.\ \cite{Yor77}) or by reduction to discrete 
time (as in \cite{vH09}).

\subsection{Markov property of the filter}

As in the introduction, we let $E$ and $F$ be Polish spaces, let 
$P:E\times\mathcal{B}(E)\to[0,1]$ and $\Phi:E\times\mathcal{B}(F)\to[0,1]$ 
be the transition kernels, and let $\mu:\mathcal{B}(E)\to[0,1]$ be the 
$P$-invariant measure defining the law of the stationary hidden Markov 
model $(X_k,Y_k)_{k\in\mathbb{Z}}$.  We denote by $\mathcal{P}(G)$ the 
space of probability measures on the Polish space $G$, endowed with the 
topology of weak convergence of probability measures.

\begin{lem}[\cite{Mey76}, Lemma 1]
\label{lem:meyer}
For $\nu\in\mathcal{P}(E)$, define the probability measure
$$
	P_\nu(A) = \int \mathbf{1}_A(x,y)\,\nu(dx')\,P(x',dx)
	\,\Phi(x,dy)\quad
	\mbox{for all }A\in\mathcal{B}(E\times F).
$$
Denote by $X:E\times F\to E$ and $Y:E\times F\to F$ the canonical 
projections.  There exists a measurable map $\Pi:\mathcal{P}(E)\times F\to
\mathcal{P}(E)$ such that $\Pi(\nu,Y)$ is a version of the 
regular conditional probability $P_\nu(X\in\,\cdot\,|Y)$ for every
$\nu\in\mathcal{P}(E)$.
\end{lem}

We now define the transition kernel
$\mathsf{\Pi}:\mathcal{P}(E)\times\mathcal{B}(\mathcal{P}(E))\to[0,1]$
as follows:
$$
	\mathsf{\Pi}(\nu,A) =
	\int \mathbf{1}_A(\Pi(\nu,y))\,\nu(dx')\,P(x',dx)\,\Phi(x,dy).
$$
We claim that the nonlinear filter $(\pi_k)_{k\ge 0}$ is a 
$\mathcal{P}(E)$-valued Markov process with transition kernel 
$\mathsf{\Pi}$.  To prove this we will need the following result on 
conditioning under a regular conditional probability due to
von Weizs{\"a}cker.

\begin{lem}[\cite{Weiz83}]
\label{lem:weiz}
Let $G$, $G'$ and $H$ be Polish spaces, and denote by $g$, $g'$ and $h$
the canonical projections from $G\times G'\times H$ on $G$, $G'$ and $H$, 
respectively.  Let $\mathbf{Q}$ be a probability measure on $G\times 
G'\times H$, and let $q_{\cdot,\cdot}:G\times 
G'\times\mathcal{B}(H)\to[0,1]$ and $q_{\cdot}:G\times\mathcal{B}(G'\times 
H)\to[0,1]$ be versions of the regular conditional probabilities
$\mathbf{Q}[h\in\,\cdot\,|g,g']$ and
$\mathbf{Q}[(g',h)\in\,\cdot\,|g]$, respectively.
Then for $\mathbf{Q}$-a.e.\ $x\in G$, the kernel
$q_{x,g'}[\,\cdot\,]$ is a version of the regular conditional probability
$q_{x}[h\in\,\cdot\,|g']$.
\end{lem}

\begin{prop}
\label{prop:markov}
For $n\ge 0$, let the nonlinear filter $\pi_n$ be a version of the regular 
conditional probability $\mathbf{P}[X_n\in\,\cdot\,|Y_1,\ldots,Y_n]$.  
Then $(\pi_k)_{k\ge 0}$ is a $\mathcal{P}(E)$-valued Markov process with
transition kernel $\mathsf{\Pi}$ and initial measure $\pi_0\sim\delta_{\mu}$.
\end{prop}

\begin{proof}
Fix $n\ge 1$. It is easily seen that for any $B\in\mathcal{B}(E\times F)$
$$
	\mathbf{P}[(X_n,Y_n)\in B|Y_1,\ldots,Y_{n-1}]=
	\int \mathbf{1}_A(x,y)\,\pi_{n-1}(dx')
	\,P(x',dx)\,\Phi(x,dy).
$$
Using Lemmas \ref{lem:weiz} and \ref{lem:meyer} and uniqueness of 
regular conditional probabilities, we find the recursive formula 
$\pi_n = \Pi(\pi_{n-1},Y_n)$ $\mathbf{P}$-a.s.  It follows easily that
$$
	\mathbf{P}[\pi_n\in A|Y_1,\ldots,Y_{n-1}] =
	\mathsf{\Pi}(\pi_{n-1},A)\quad\mathbf{P}\mbox{-a.s.}
	\quad\mbox{for all }A\in\mathcal{B}(\mathcal{P}(E)),
$$
completing the proof.
\end{proof}

We now establish the two elementary facts stated in the introduction.

\begin{lem}
Let $\mathsf{m}\in\mathcal{P}(\mathcal{P}(E))$ be any 
$\mathsf{\Pi}$-invariant probability measure.  Then the barycenter of
$\mathsf{m}$ is a $P$-invariant probability measure.
\end{lem}

\begin{proof}
Let $m\in\mathcal{P}(E)$ be the barycenter of $\mathsf{m}$.  By 
definition,
$$
	m(A) = \int \nu(A)\,\mathsf{m}(d\nu) =
	\int \nu(A)\,\mathsf{\Pi}(\nu',d\nu)\,\mathsf{m}(d\nu')
	\quad\mbox{for }A\in\mathcal{B}(E).
$$
But note that $\int \nu(A)\,\mathsf{\Pi}(\nu',d\nu) =
\mathbf{E}_{P_{\nu'}}[P_{\nu'}(X\in A|Y)]=
\int P(x,A)\,\nu'(dx)$ by the definition of $\mathsf{\Pi}$.
It follows directly that $mP=m$, that is, $m$ is $P$-invariant.
\end{proof}

\begin{lem}
\label{lem:exist}
There is at least one $\mathsf{\Pi}$-invariant measure
with barycenter $\mu$.
\end{lem}

\begin{proof}
For $n\in\mathbb{Z}$, let $\tilde\pi_n$ be a version of the regular 
conditional probability 
$\mathbf{P}[X_n\in\,\cdot\,|\mathcal{F}_{-\infty,n}^Y]$.  
Proceeding exactly as in the proof of Proposition \ref{prop:markov}, we 
find that $(\tilde\pi_k)_{k\in\mathbb{Z}}$ is a $\mathcal{P}(E)$-valued 
Markov process with transition kernel $\mathsf{\Pi}$.  But as the 
underlying hidden Markov model $(X_k,Y_k)_{k\in\mathbb{Z}}$ is stationary, 
clearly $(\tilde\pi_k)_{k\in\mathbb{Z}}$ is also stationary.  Therefore, 
the law of $\tilde\pi_0$ is a $\mathsf{\Pi}$-invariant measure, and its
barycenter is $\mu$ by the tower property of the conditional expectation.
\end{proof}

\subsection{Proof of Theorem \ref{thm:exchg}: sufficiency}

The proof is essentially contained in Kunita \cite{Kun71}, though we are 
careful here not to exploit any unnecessary assumptions.  The idea is to 
introduce a suitable randomization, which is most conveniently done in a 
canonical probability model.  To this end, define the Polish 
space $\Omega_0=\mathcal{P}(E)\times E\times (E\times F)^{\mathbb{N}}$ 
with the canonical projections $m_0:\Omega_0\to\mathcal{P}(E)$ 
and (with a slight abuse of notation) $X_0:\Omega_0\to E$,
$(X_k,Y_k)_{k\ge 1}:\Omega_0\to (E\times F)^{\mathbb{N}}$.  
Given $\mathsf{m}\in\mathcal{P}(\mathcal{P}(E))$, 
we define a probability measure $\mathbf{P}_{\mathsf{m}}$ on $\Omega_0$ 
with the finite dimensional distributions
\begin{multline*}
	\mathbf{P}_{\mathsf{m}}((m_0,X_0,\ldots,X_n,Y_1,\ldots,Y_n)\in A) 
	= \mbox{}\\
	\int \mathbf{1}_A(\nu,x_0,\ldots,x_n,y_1,\ldots,y_n)\,
	\nu(dx_0)\,P(x_0,dx_1)\,\Phi(x_1,dy_1)\cdots \mbox{}\\
	P(x_{n-1},dx_n)
	\,\Phi(x_n,dy_n)\,\mathsf{m}(d\nu).
\end{multline*}
We now define for $n\ge 0$ three distinguished nonlinear filters:
\begin{equation*}
\begin{array}{ll}
	\pi_n^{\rm min} & \mbox{} := \mathbf{P}_{\mathsf{m}}[X_n\in\,\cdot\,|
	Y_1,\ldots,Y_n],\\
	\pi_n^{\mathsf{m}} & \mbox{} := \mathbf{P}_{\mathsf{m}}[X_n\in\,\cdot\,|
	m_0,Y_1,\ldots,Y_n],\\
	\pi_n^{\rm max} & \mbox{} := \mathbf{P}_{\mathsf{m}}[X_n\in\,\cdot\,|
	m_0,X_0,Y_1,\ldots,Y_n].
\end{array}
\end{equation*}
We now have the following easy result. Here 
$\delta_\mu,\varepsilon_\mu\in\mathcal{P}(\mathcal{P}(E))$ are defined by
$\delta_\mu(A)=\mathbf{1}_{\mu\in A}$ (as usual) and
$\varepsilon_\mu(A) = \int\mathbf{1}_{\delta_x\in A}\,\mu(dx)$.

\begin{lem}
\label{lem:randomiz}
Let $\mathsf{m}\in\mathcal{P}(\mathcal{P}(E))$ be any probability measure
with barycenter $\mu$.  Then $(\pi_n^{\rm min})_{n\ge 0}$,
$(\pi_n^{\mathsf{m}})_{n\ge 0}$,
$(\pi_n^{\rm max})_{n\ge 0}$ are $\mathcal{P}(E)$-valued Markov processes 
under $\mathbf{P}_{\mathsf{m}}$ with transition kernel $\mathsf{\Pi}$ and
initial measures $\delta_\mu$, $\mathsf{m}$, $\varepsilon_\mu$, 
respectively.
\end{lem}

\begin{proof}
The proof is identical to that of Proposition \ref{prop:markov}.
\end{proof}

The following result completes the proof of sufficiency.

\begin{prop}
Let $p\in\mathbb{N}$, let $f_i:E\to\mathbb{R}$,
$i=1,\ldots,p$ be bounded measurable functions, and let
$\kappa:\mathbb{R}^p\to\mathbb{R}$ be convex.
Define the bounded measurable function $F:\mathcal{P}(E)\to\mathbb{R}$ as
$F(\nu) = \kappa\left(\textstyle{\int f_1(x)\,\nu(dx),\ldots,\int 
f_p(x)\,\nu(dx)}\right)$.  Finally, let 
$\mathsf{m}\in\mathcal{P}(\mathcal{P}(E))$ be any
$\mathsf{\Pi}$-invariant measure with barycenter $\mu$.  Then
\begin{multline*}
	\mathbf{E}\left[
	\kappa\left(\mathbf{E}[f_1(X_0)|\mathcal{F}^Y_{-\infty,0}],\ldots,
	\mathbf{E}[f_p(X_0)|\mathcal{F}^Y_{-\infty,0}]\right)\right]
	\le \int F(\nu)\,\mathsf{m}(d\nu) \\ \mbox{}\le
	\mathbf{E}\left[
	\kappa\left(\mathbf{E}[f_1(X_0)|\mathcal{G}^{\phantom{Y}}_{-\infty,0}],\ldots,
	\mathbf{E}[f_p(X_0)|\mathcal{G}^{\phantom{Y}}_{-\infty,0}]\right)\right],
\end{multline*}
where $\mathcal{G}_{-\infty,0}:=\bigcap_n(\mathcal{F}^Y_{-\infty,0}
\vee\mathcal{F}^X_{-\infty,n})$.  In particular, if (\ref{eq:exchg}) 
holds, $\mathsf{m}$ coincides with the distinguished 
$\mathsf{\Pi}$-invariant measure defined in the proof of Lemma \ref{lem:exist}.
\end{prop}

\begin{proof}
Note that as $\kappa$ is convex, it is continuous, hence $F$ is bounded 
and measurable.  It is an immediate consequence of Jensen's inequality 
that
$$
	\mathbf{E}_{\mathsf{m}}[F(\pi_n^{\rm min})] \le
	\mathbf{E}_{\mathsf{m}}[F(\pi_n^{\mathsf{m}})] =
	\int F(\nu)\,\mathsf{m}(d\nu) \le
	\mathbf{E}_{\mathsf{m}}[F(\pi_n^{\rm max})]
$$
for every $n\ge 0$, where we have used Lemma \ref{lem:randomiz} and the 
$\mathsf{\Pi}$-invariance of $\mathsf{m}$ to obtain the middle equality.
Using Lemma \ref{lem:randomiz} and the stationarity of 
$(X_k,Y_k)_{k\in\mathbb{Z}}$ under $\mathbf{P}$, it is also easily seen 
that the laws of $\pi_n^{\rm min}(f)$, $\pi_n^{\rm max}(f)$ under
$\mathbf{P}_{\mathsf{m}}$ coincide with the laws of
$\mathbf{E}[f(X_0)|Y_{-n+1},\ldots,Y_0]$,
$\mathbf{E}[f(X_0)|X_{-n},Y_{-n+1},\ldots,Y_0]$ under $\mathbf{P}$,
respectively.  We therefore have for every $n\ge 0$
\begin{multline*}
	\mathbf{E}\left[
	\kappa\left(\mathbf{E}[f_1(X_0)|\mathcal{F}^Y_{-n+1,0}],\ldots,
	\mathbf{E}[f_p(X_0)|\mathcal{F}^Y_{-n+1,0}]\right)\right]
	\le \int F(\nu)\,\mathsf{m}(d\nu) \\ \mbox{}\le
	\mathbf{E}\left[
	\kappa\left(\mathbf{E}[f_1(X_0)|
	\mathcal{G}^{\phantom{Y}}_{-n,0}],\ldots,
	\mathbf{E}[f_p(X_0)|
	\mathcal{G}^{\phantom{Y}}_{-n,0}]\right)\right],
\end{multline*}
where $\mathcal{G}_{-n,0}:=\mathcal{F}^Y_{-\infty,0}
\vee\mathcal{F}^X_{-\infty,-n}$ and we have used the fact that
$$
	\mathbf{E}[f(X_0)|X_{-n},Y_{-n+1},\ldots,Y_0]=
	\mathbf{E}[f(X_0)|\mathcal{G}_{-n,0}]\quad
	\mathbf{P}\mbox{-a.s.}
$$ 
as $\mathcal{F}_{-n+1,0}^X\vee\mathcal{F}_{-n+1,0}^Y$ is conditionally
independent of $\mathcal{F}_{-\infty,-n-1}^X\vee\mathcal{F}_{-\infty,-n}^Y$
given $X_{-n}$.  But as $\kappa$ is continuous, the equation display in 
the statement of the result follows by letting $n\to\infty$ using
the martingale convergence theorem.

Now suppose that (\ref{eq:exchg}) holds, and denote by $\mathsf{m}_0$ be 
the distinguished $\mathsf{\Pi}$-invariant measure obtained in the proof 
of Lemma \ref{lem:exist}.  Then we have evidently shown that
$\int F(\nu)\,\mathsf{m}(d\nu)=\int F(\nu)\,\mathsf{m}_0(d\nu)$ for all 
functions $F$ of the form $F(\nu) = \kappa\left(\textstyle{\int 
f_1(x)\,\nu(dx),\ldots,\int f_p(x)\,\nu(dx)}\right)$ for any $p$, bounded 
measurable $f_1,\ldots,f_p$ and convex $\kappa$.  We claim that this class 
of functions is measure-determining, so we can conclude that 
$\mathsf{m}=\mathsf{m}_0$.  To establish the claim, first note that by the 
Stone-Weierstrass theorem, any continuous function on $\mathbb{R}^p$ can 
be approximated uniformly on any compact set by the difference of convex 
functions.  As $f_1,\ldots,f_p$ are bounded (hence take values in a 
compact subset of $\mathbb{R}^p$), it therefore suffices to assume that 
$\kappa$ is continuous rather than convex.  Next, note that the indicator 
function $\mathbf{1}_A$ of any open subset $A$ of $\mathbb{R}^p$
can be obtained as the increasing limit of nonnegative continuous 
functions.  It therefore suffices to assume that $\kappa$ is the indicator 
of an open subset of $\mathbb{R}^p$.  But any probability measure on a 
Polish space is regular, so it suffices to assume that $\kappa$ is the 
indicator function of a Borel subset of $\mathbb{R}^p$.  The proof is 
completed by an application of the Dynkin system lemma.
\end{proof}

\subsection{Proof of Theorem \ref{thm:exchg}: necessity}

We will in fact prove necessity under a weaker assumption than stated 
in the theorem: the key assumption is 
\begin{equation}
\label{eq:faspt}
	\bigcap_{n\le 0}\big(\mathcal{F}_{-\infty,k}^Y\vee
	\mathcal{F}_{-\infty,n}^X\big) =
	\mathcal{F}_{1,k}^Y\vee
	\bigcap_{n\le 0}\big(\mathcal{F}_{-\infty,0}^Y\vee
	\mathcal{F}_{-\infty,n}^X\big)\quad
	\mathbf{P}\mbox{-a.s.}\quad
	\forall\,k\in\mathbb{N}.
\end{equation}
The assumption in the theorem that $\Phi$ possesses a transition density 
only enters the proof inasmuch as it guarantees the validity (\ref{eq:faspt}).
Let us note that the assumption of the theorem is itself weaker than 
nondegeneracy of the observations, as the transition density is not 
required to be strictly positive here.

\begin{lem}
Suppose there exists a $\sigma$-finite reference measure $\varphi$ on $F$ 
and a transition density $g:E\times F\to [0,\infty[\mbox{}$ such that
$\Phi(x,A) = \int \mathbf{1}_A(y)\,g(x,y)\,\varphi(dy)$ for all $x\in E$, 
$A\in\mathcal{B}(F)$.  Then the identity (\ref{eq:faspt}) holds true.
\end{lem}

\begin{proof}
It is easily seen that the assumption guarantees the existence of a 
probability measure $\mathbf{Q}$ such that $\mathbf{P}\ll\mathbf{Q}$
and $\mathcal{F}_{1,k}^Y$ is independent of 
$\mathcal{F}_{-\infty,0}^X\vee\mathcal{F}_{-\infty,0}^Y$ under 
$\mathbf{Q}$.  Thus the identity in (\ref{eq:faspt}) holds 
$\mathbf{Q}$-a.s., and therefore $\mathbf{P}$-a.s.
\end{proof}

The proof is based on the following result.

\begin{lem}
\label{lem:kzero}
Suppose there exists a unique $\mathsf{\Pi}$-invariant measure
with barycenter $\mu$ and that assumption (\ref{eq:faspt}) holds.  
Then we have for every $A\in\mathcal{B}(E)$
$$
	\mathbf{P}\big[X_0\in A\big|
	\textstyle{\bigcap_{n}\big(\mathcal{F}_{-\infty,0}^Y
	\vee\mathcal{F}_{-\infty,n}^X}\big)\big] =
	\mathbf{P}\big[X_0\in A\big|\mathcal{F}_{-\infty,0}^Y\big] 
	\quad\mathbf{P}\mbox{-a.s.}
$$
\end{lem}

\begin{proof}
Define the regular conditional probabilities
$\pi_k^0=\mathbf{P}[X_k\in\,\cdot\,|\mathcal{F}_{-\infty,k}^Y]$ and
$\pi_k^1=\mathbf{P}[X_k\in\,\cdot\,|\bigcap_n(\mathcal{F}_{-\infty,k}^Y
\vee\mathcal{F}_{-\infty,n}^X)]$, and denote by
$\mathsf{m}_0,\mathsf{m}_1\in\mathcal{P}(\mathcal{P}(E))$ the laws of 
$\pi_0^0$ and $\pi_0^1$, respectively.  Then $\mathsf{m}_0$ is the 
$\mathsf{\Pi}$-invariant measure defined in the proof of Lemma 
\ref{lem:exist}.  We claim that $\mathsf{m}_1$ is also 
$\mathsf{\Pi}$-invariant.  Indeed, this follows as a variant of Lemma 
\ref{lem:weiz} (pp.\ 95--96 in \cite{Weiz83}) and the assumption 
(\ref{eq:faspt}) imply that $\pi_k^1=\Pi(\pi_{k-1}^1,Y_k)$ 
$\mathbf{P}$-a.s., so that $(\pi_k^1)_{k\in\mathbb{Z}}$ is Markov with 
transition kernel $\mathsf{\Pi}$, while $(\pi_k^1)_{k\in\mathbb{Z}}$ is 
easily seen to be a stationary process.

Clearly $\mathsf{m}_0$ and $\mathsf{m}_1$ both have barycenter $\mu$, so 
by assumption $\mathsf{m}_0=\mathsf{m}_1$.  Thus
$$
	\mathbf{E}[(\pi_k^1(A)-\pi_k^0(A))^2] =
	\mathbf{E}[(\pi_k^1(A))^2] - \mathbf{E}[(\pi_k^0(A))^2] = 
	\mathsf{m}_1(F_A)-\mathsf{m}_0(F_A) = 0
$$
for every $A\in\mathcal{B}(E)$, where we defined
$F_A:\nu\mapsto (\nu(A))^2$.  It follows that
$\pi_0^1(A)=\pi_0^0(A)$ $\mathbf{P}$-a.s.\ for every $A\in\mathcal{B}(E)$,
which completes the proof.
\end{proof}

To complete the proof, we require the following easy variant of 
Lemma \ref{lem:meyer}.

\begin{lem}
\label{lem:smooth}
For $\nu\in\mathcal{P}(E)$ and $k\in\mathbb{N}$, define the probability 
measure
\begin{multline*}
	P_\nu^k(A) = \int \mathbf{1}_A(x_0,y_1,\ldots,y_k)\,
	\nu(dx_0)\,P(x_0,dx_1)\,\Phi(x_1,dy_1) \cdots\mbox{}\\
	P(x_{k-1},dx_k)\,\Phi(x_k,dy_k)
	\quad\mbox{for }A\in\mathcal{B}(E\times F^k).
\end{multline*}
Denote by $X:E\times F^k\to E$ and 
$Y^k:E\times F^k\to F^k$ the canonical projections.  There exists a 
measurable map $\Sigma^k:\mathcal{P}(E)\times F^k\to
\mathcal{P}(E)$ such that $\Sigma^k(\nu,Y^k)$ is a version of the 
regular conditional probability $P_\nu^k(X\in\,\cdot\,|Y^k)$ for every
$\nu\in\mathcal{P}(E)$.
\end{lem}

We now complete the proof.

\begin{prop}
Suppose there exists a unique $\mathsf{\Pi}$-invariant measure with 
barycenter $\mu$ and that assumption (\ref{eq:faspt}) holds.  Then 
(\ref{eq:exchg}) holds true.
\end{prop}

\begin{proof}
As $\bigcup_{k\le 0}L^1(\mathcal{F}_{k,0}^X\vee\mathcal{F}_{k,0}^Y,
\mathbf{P})$ is dense in 
$L^1(\mathcal{F}_{-\infty,0}^X\vee\mathcal{F}_{-\infty,0}^Y,\mathbf{P})$,
it suffices to show that for every $k\le 0$ and
$Z\in L^1(\mathcal{F}_{k,0}^X\vee\mathcal{F}_{k,0}^Y,\mathbf{P})$
$$
	\mathbf{E}\big[Z\big|
	\textstyle{\bigcap_{n}\big(\mathcal{F}_{-\infty,0}^Y
	\vee\mathcal{F}_{-\infty,n}^X}\big)\big] =
	\mathbf{E}\big[Z\big|\mathcal{F}_{-\infty,0}^Y\big] 
	\quad\mathbf{P}\mbox{-a.s.}
$$
However, for $Z\in 
L^1(\mathcal{F}_{k,0}^X\vee\mathcal{F}_{k,0}^Y,\mathbf{P})$, we have
by the Markov property
$$
	\mathbf{E}\big[Z\big|
	\textstyle{\bigcap_{n}\big(\mathcal{F}_{-\infty,0}^Y
	\vee\mathcal{F}_{-\infty,n}^X}\big)\big] =
	\mathbf{E}\big[\,
	\mathbf{E}[Z|\sigma\{X_k\}\vee\mathcal{F}_{k,0}^Y]\,
	\big|
	\textstyle{\bigcap_{n}\big(\mathcal{F}_{-\infty,0}^Y
	\vee\mathcal{F}_{-\infty,n}^X}\big)\big].
$$
It therefore suffices to consider $Z\in 
L^1(\sigma\{X_k\}\vee\mathcal{F}_{k,0}^Y,\mathbf{P})$.
But note that the class of random variables
$\{Z^XZ^Y:Z^X\in L^\infty(\sigma\{X_k\},\mathbf{P}),~Z^Y\in 
L^\infty(\mathcal{F}_{k,0}^Y,\mathbf{P})\}$ is total in
$L^1(\sigma\{X_k\}\vee\mathcal{F}_{k,0}^Y,\mathbf{P})$.
Therefore, it suffices to show that
$$
	\mathbf{P}\big[X_k\in A\big|
	\textstyle{\bigcap_{n}\big(\mathcal{F}_{-\infty,0}^Y
	\vee\mathcal{F}_{-\infty,n}^X}\big)\big] =
	\mathbf{P}\big[X_k\in A\big|\mathcal{F}_{-\infty,0}^Y\big] 
	\quad\mathbf{P}\mbox{-a.s.}	
$$
for all $k\le 0$ and $A\in\mathcal{B}(E)$.  For $k=0$, this follows
directly from Lemma \ref{lem:kzero}.

For $k<0$, we proceed as follows.  Define $\pi_k^0$ and $\pi_k^1$ as in
the proof of Lemma \ref{lem:kzero}.  It is easily established using Lemma 
\ref{lem:weiz} that
$$
	\mathbf{P}\big[X_k\in\,\cdot\,\big|\mathcal{F}_{-\infty,0}^Y\big] =
	\Sigma^k(\pi_k^0,Y_{k+1},\ldots,Y_0)\quad
	\mathbf{P}\mbox{-a.s.}
$$
Similarly, a variant of Lemma \ref{lem:weiz} (pp.\ 95--96 in 
\cite{Weiz83}) and (\ref{eq:faspt}) imply
$$
	\mathbf{P}\big[X_k\in\,\cdot\,\big|
	\textstyle{\bigcap_{n}\big(\mathcal{F}_{-\infty,0}^Y
	\vee\mathcal{F}_{-\infty,n}^X}\big)\big] =
	\Sigma^k(\pi_k^1,Y_{k+1},\ldots,Y_0)\quad
	\mathbf{P}\mbox{-a.s.}
$$
But by Lemma \ref{lem:kzero}, applying the Dynkin system 
lemma with a countable generating system, and using that
$(X_k,Y_k)_{k\in\mathbb{Z}}$ is stationary under $\mathbf{P}$,
it follows directly that $\pi_k^0=\pi_k^1$ $\mathbf{P}$-a.s.  
This completes the proof.
\end{proof}

\bibliographystyle{acmtrans-ims}
\bibliography{ref}

\begin{thebibliography}{}
\ifx \url   \undefined \def \url#1{#1}   \fi

\bibitem{BCL04}
\textsc{Baxendale, P.}, \textsc{Chigansky, P.}, \textsc{and} \textsc{Liptser,
  R.} (2004).
\newblock Asymptotic stability of the {W}onham filter: ergodic and nonergodic
  signals.
\newblock \emph{SIAM J. Control Optim.\/}~\textbf{43},~2, 643--669.

\bibitem{BL07}
\textsc{Brossard, J.} \textsc{and} \textsc{Leuridan, C.} (2007).
\newblock Cha\^\i nes de {M}arkov constructives index\'ees par {$\bold Z$}.
\newblock \emph{Ann. Probab.\/}~\textbf{35},~2, 715--731.

\bibitem{Bud03}
\textsc{Budhiraja, A.} (2003).
\newblock Asymptotic stability, ergodicity and other asymptotic properties of
  the nonlinear filter.
\newblock \emph{Ann. Inst. H. Poincar\'e Probab. Statist.\/}~\textbf{39},~6,
  919--941.

\bibitem{CMR05}
\textsc{Capp{\'e}, O.}, \textsc{Moulines, E.}, \textsc{and} \textsc{Ryd{\'e}n,
  T.} (2005).
\newblock \emph{Inference in hidden {M}arkov models}.
\newblock Springer Series in Statistics. Springer, New York.

\bibitem{CY03}
\textsc{Chaumont, L.} \textsc{and} \textsc{Yor, M.} (2003).
\newblock \emph{Exercises in probability}.
\newblock Cambridge University Press, Cambridge.

\bibitem{CvH10}
\textsc{Chigansky, P.} \textsc{and} \textsc{van Handel, R.} (2010).
\newblock A complete solution to {B}lackwell's unique ergodicity problem for
  hidden {M}arkov chains.
\newblock \emph{Ann. Appl. Probab.\/}~\textbf{20},~6, 2318--2345.

\bibitem{CLP07}
\textsc{Crimaldi, I.}, \textsc{Letta, G.}, \textsc{and} \textsc{Pratelli, L.}
  (2007).
\newblock Sur l'interversion de l'ordre entre deux op\'erations sur les tribus.
\newblock \emph{C. R. Math. Acad. Sci. Paris\/}~\textbf{345},~6, 341--344.

\bibitem{HS06}
\textsc{den Hollander, F.} \textsc{and} \textsc{Steif, J.~E.} (2006).
\newblock Random walk in random scenery: a survey of some recent results.
\newblock In \emph{Dynamics \& stochastics}. IMS Lecture Notes Monogr. Ser.,
  Vol.~\textbf{48}. Inst. Math. Statist., Beachwood, OH, 53--65.

\bibitem{ES01}
\textsc{{\'E}mery, M.} \textsc{and} \textsc{Schachermayer, W.} (2001).
\newblock On {V}ershik's standardness criterion and {T}sirelson's notion of
  cosiness.
\newblock In \emph{S\'eminaire de {P}robabilit\'es, {XXXV}}. Lecture Notes in
  Math., Vol. \textbf{1755}. Springer, Berlin, 265--305.

\bibitem{Kal82}
\textsc{Kalikow, S.~A.} (1982).
\newblock {$T,\,T^{-1}$} transformation is not loosely {B}ernoulli.
\newblock \emph{Ann. of Math. (2)\/}~\textbf{115},~2, 393--409.

\bibitem{Kes98}
\textsc{Kesten, H.} (1998).
\newblock Distinguishing and reconstructing sceneries from observations along
  random walk paths.
\newblock In \emph{Microsurveys in discrete probability ({P}rinceton, {NJ},
  1997)}. DIMACS Ser. Discrete Math. Theoret. Comput. Sci., Vol.~\textbf{41}.
  Amer. Math. Soc., Providence, RI, 75--83.

\bibitem{Kun71}
\textsc{Kunita, H.} (1971).
\newblock Asymptotic behavior of the nonlinear filtering errors of {M}arkov
  processes.
\newblock \emph{J. Multivariate Anal.\/}~\emph{1}, 365--393.

\bibitem{Lau10}
\textsc{Laurent, S.} (2010).
\newblock Further comments on the representation problem for stationary
  processes.
\newblock \emph{Statist. Probab. Lett.\/}~\textbf{80},~7-8, 592--596.

\bibitem{Maj10}
\textsc{Majda, A.~J.}, \textsc{Harlim, J.}, \textsc{and} \textsc{Gershgorin,
  B.} (2010).
\newblock Mathematical strategies for filtering turbulent dynamical systems.
\newblock \emph{Discrete Contin. Dyn. Syst.\/}~\textbf{27},~2, 441--486.

\bibitem{Mas66}
\textsc{Masani, P.} (1966).
\newblock Wiener's contributions to generalized harmonic analysis, prediction
  theory and filter theory.
\newblock \emph{Bull. Amer. Math. Soc.\/}~\emph{72}, 73--125.

\bibitem{MR03}
\textsc{Matzinger, H.} \textsc{and} \textsc{Rolles, S. W.~W.} (2003).
\newblock Reconstructing a random scenery observed with random errors along a
  random walk path.
\newblock \emph{Probab. Theory Related Fields\/}~\textbf{125},~4, 539--577.

\bibitem{Mei74}
\textsc{Meilijson, I.} (1974).
\newblock Mixing properties of a class of skew-products.
\newblock \emph{Israel J. Math.\/}~\emph{19}, 266--270.

\bibitem{Mey76}
\textsc{Meyer, P.~A.} (1976).
\newblock La th\'eorie de la pr\'ediction de {F}. {K}night.
\newblock In \emph{S\'eminaire de {P}robabilit\'es, {X} ({P}remi\`ere partie,
  {U}niv. {S}trasbourg, {S}trasbourg, ann\'ee universitaire 1974/1975)}.
  Springer, Berlin, 86--103. Lecture Notes in Math., Vol. 511.

\bibitem{MT09}
\textsc{Meyn, S.} \textsc{and} \textsc{Tweedie, R.~L.} (2009).
\newblock \emph{Markov chains and stochastic stability}, Second ed.
\newblock Cambridge University Press, Cambridge.

\bibitem{Orn73}
\textsc{Ornstein, D.~S.} (1973).
\newblock An application of ergodic theory to probability theory.
\newblock \emph{Ann. Probability\/}~\textbf{1},~1, 43--65.

\bibitem{Rah78}
\textsc{Rahe, M.} (1978).
\newblock Relatively finitely determined implies relatively very weak
  {B}ernoulli.
\newblock \emph{Canad. J. Math.\/}~\textbf{30},~3, 531--548.

\bibitem{Sin89}
\textsc{Sinai, Y.~G.} (1989).
\newblock Kolmogorov's work on ergodic theory.
\newblock \emph{Ann. Probab.\/}~\textbf{17},~3, 833--839.

\bibitem{Smo71}
\textsc{Smorodinsky, M.} (1971).
\newblock \emph{Ergodic theory, entropy}.
\newblock Lecture Notes in Mathematics, Vol. 214. Springer-Verlag, Berlin.

\bibitem{Tot70}
\textsc{Totoki, H.} (1970).
\newblock On a class of special flows.
\newblock \emph{Z. Wahrscheinlichkeitstheorie und Verw. Gebiete\/}~\emph{15},
  157--167.

\bibitem{vH09}
\textsc{van Handel, R.} (2009a).
\newblock The stability of conditional {M}arkov processes and {M}arkov chains
  in random environments.
\newblock \emph{Ann. Probab.\/}~\textbf{37},~5, 1876--1925.

\bibitem{vHipf}
\textsc{van Handel, R.} (2009b).
\newblock Uniform time average consistency of {M}onte {C}arlo particle filters.
\newblock \emph{Stochastic Process. Appl.\/}~\textbf{119},~11, 3835--3861.

\bibitem{Van10sic}
\textsc{van Handel, R.} (2009/10).
\newblock When do nonlinear filters achieve maximal accuracy?
\newblock \emph{SIAM J. Control Optim.\/}~\textbf{48},~5, 3151--3168.

\bibitem{VR59}
\textsc{Volkonskii, V.~A.} \textsc{and} \textsc{Rozanov, Y.~A.} (1959).
\newblock Some limit theorems for random functions. {I}.
\newblock \emph{Theor. Probability Appl.\/}~\emph{4}, 178--197.

\bibitem{Weiz83}
\textsc{von Weizs{\"a}cker, H.} (1983).
\newblock Exchanging the order of taking suprema and countable intersections of
  {$\sigma $}-algebras.
\newblock \emph{Ann. Inst. H. Poincar\'e Sect. B (N.S.)\/}~\textbf{19},~1,
  91--100.

\bibitem{Wei72}
\textsc{Weiss, B.} (1972).
\newblock The isomorphism problem in ergodic theory.
\newblock \emph{Bull. Amer. Math. Soc.\/}~\emph{78}, 668--684.

\bibitem{Wil91}
\textsc{Williams, D.} (1991).
\newblock \emph{Probability with martingales}.
\newblock Cambridge University Press.

\bibitem{YY11}
\textsc{Yano, K.} \textsc{and} \textsc{Yor, M.} (2011).
\newblock Around {T}sirelson's equation, or: the evolution process may not
  explain everything.
\newblock preprint, arxiv:0906.3442.

\bibitem{Yor77}
\textsc{Yor, M.} (1977).
\newblock Sur les th\'eories du filtrage et de la pr\'ediction.
\newblock In \emph{S\'eminaire de {P}robabilit\'es, {XI} ({U}niv. {S}trasbourg,
  {S}trasbourg, 1975/1976)}. Springer, Berlin, 257--297. Lecture Notes in
  Math., Vol. 581.

\end{thebibliography}

\end{document}